\documentclass[11pt, a4paper]{scrartcl}

\usepackage[margin=1in]{geometry}

\usepackage[english]{babel}
\usepackage{graphicx,xcolor}
\usepackage{amsmath, amssymb, amsthm}
\usepackage{mathtools, bm}
\usepackage[ISO]{diffcoeff}
\usepackage{caption}
\usepackage{braket}
\usepackage{subcaption}
\usepackage{enumitem}
\usepackage{hyperref}
\usepackage[autostyle=try,english=american,german=quotes]{csquotes}
\usepackage{pgfplots}
\pgfplotsset{compat=1.18, width=\textwidth}
\usetikzlibrary{plotmarks}

\DeclarePairedDelimiterX{\abs}[1]{\lvert}{\rvert}{#1}
\DeclarePairedDelimiterX{\norm}[1]{\lVert}{\rVert}{#1}
\DeclarePairedDelimiterX{\dotprod}[2]{\langle}{\rangle}{#1,#2}

\definecolor{wongblue}{HTML}{0072B2}
\definecolor{wongorange}{HTML}{E69F00}
\definecolor{wonggreen}{HTML}{009E73}
\definecolor{wongpurple}{HTML}{CC79A7}
\definecolor{wonglblue}{HTML}{56B4E9}
\definecolor{wongbrown}{HTML}{D55E00}
\definecolor{wongyellow}{HTML}{F0E442}

\pgfplotscreateplotcyclelist{wonglist}{
    {wongblue,mark=*,thick,mark size=2pt},
    {wongorange,mark=square*,thick,mark size=2pt},
    {wonggreen,mark=triangle*,thick,mark size=3pt},
    {wongpurple,mark=+,thick,mark size=3pt},
    {wonglblue,mark=diamond*,thick,mark size=3pt},
    {wongbrown,mark=x,thick,mark size=3pt},
    {wongyellow,mark=pentagon*,thick,mark size=3pt}
}

\newcommand{\R}{\mathbb{R}}
\newcommand{\bvec}[1]{\bm{#1}}
\newcommand{\bmat}[1]{\bm{#1}}
\renewcommand{\S}{\mathcal{S}}
\newcommand{\N}{\mathcal{N}}
\newcommand{\dt}{\Delta t}

\newcommand{\halfwidth}{0.47\linewidth}

\usepackage[noabbrev]{cleveref}
\theoremstyle{definition}
\newtheorem{definition}{Definition}[section]
\newtheorem{lemma}[definition]{Lemma}
\newtheorem{theorem}[definition]{Theorem}
\newtheorem{corollary}[definition]{Corollary}

\newtheorem{remark}[definition]{Remark}

\crefname{figure}{figure}{figures}
\crefname{definition}{definition}{definitions}
\crefname{theorem}{theorem}{theorems}
\crefname{lemma}{lemma}{lemmas}
\crefname{corollary}{corollary}{corollaries}
\crefname{proposition}{proposition}{propositions}
\crefname{remark}{remark}{remarks}
\crefname{example}{example}{examples}

\graphicspath{{figures/}}
\newcommand{\email}[1]{\protect\href{mailto:#1}{#1}}

\author{
    Heinrich Kraus\thanks{Universität Kassel, Heinrich-Plett-Straße 40, Kassel 34132, Germany (\email{heinrich.kraus@itwm.fraunhofer.de}).}
\and
    Jörg Kuhnert\thanks{Fraunhofer ITWM, Fraunhofer-Platz 1, Kaiserslautern 67663, Germany (\email{joerg.kuhnert@itwm.fraunhofer.de}).}
\and
    Pratik Suchde\thanks{University of Luxembourg, 2, avenue de l'Université, 4365 Esch-sur-Alzette, Luxembourg (\email{pratik.suchde@gmail.com}).}
}

\title{Higher-Order Generalized Finite Differences for Variable Coefficient Diffusion Operators\thanks{Preprint}}

\begin{document}
\maketitle
\begin{abstract}
    We present a novel approach of discretizing variable coefficient diffusion operators in the context of meshfree generalized finite difference methods.
Our ansatz uses properties of derived operators and combines the discrete Laplace operator with reconstruction functions approximating the diffusion coefficient.
Provided that the reconstructions are of a sufficiently high order, we prove that the order of accuracy of the discrete Laplace operator transfers to the derived diffusion operator.
We show that the new discrete diffusion operator inherits the diagonal dominance property of the discrete Laplace operator.
Finally, we present the possibility of discretizing anisotropic diffusion operators with the help of derived operators.
Our numerical results for Poisson's equation and the heat equation show that even low-order reconstructions preserve the order of the underlying discrete Laplace operator for sufficiently smooth diffusion coefficients.
In experiments, we demonstrate the applicability of the new discrete diffusion operator to interface problems with point clouds not aligning to the interface and numerically show first-order convergence.

\end{abstract}
\section{Introduction and governing equations}
Meshfree methods grew in popularity due to their applicability to problems where classical mesh-based methods struggle due to the high computational cost of meshing and re-meshing algorithms.
One of the most popular meshfree methods is \emph{Smoothed Particle Hydrodynamics} (SPH) that has been originally developed for solving astrophysical problems, but later has been used for general fluid mechanical problems \cite{Monaghan_1992}.
However, setting boundary conditions can be challenging and the SPH method is not consistent.
In this work, we present a generalized finite difference method that stems from a generalization of the SPH method.
It expands the capabilities of SPH for free surface problems by consistent discrete differential operators from a moving least squares (MLS) finite difference ansatz \cite{Kuhnert_1999}.
This method proved itself in practice in the simulation of free surface flows with complex geometries in industrial applications \cite{MESHFREE,Michel_Seifarth_Kuhnert_Suchde_2021,Veltmaat_Mehrens_Endres_Kuhnert_Suchde_2022}.
Our eventual goal is to extend this method to allow phase change simulations in a one-fluid model where the phase change occurs across a diffuse phase change region and not at a sharp interface dividing the domain into phases \cite{Swaminathan_Voller_1993,Saucedo-Zendejo_Resendiz-Flores_2019}.
In such models, material properties can exhibit jumps or large gradients in the phase change region.
Such material properties appear in variable coefficient diffusion operators that we abbreviate by
\begin{equation} \label{eq:diffusion_operator}
    \nabla\cdot(\bmat{\kappa}\nabla) u := \nabla\cdot(\bmat{\kappa}\nabla u).
\end{equation}
In the general case the diffusion coefficient $\bmat{\kappa}$ is a matrix-valued function and the respective differential operator is also called anisotropic diffusion operator.

For the most part of the paper we will consider the case $\bmat{\kappa} = \kappa\bmat{I}$ with a scalar-valued diffusivity $\kappa$, and the diffusion operator $\nabla\cdot(\kappa\nabla)u=\nabla\cdot(\kappa\nabla u)$.
We always assume that there exist positive constants $\kappa_{\min}$, $\kappa_{\max}\in\mathbb{R}^+$ that bound the diffusivity according to $0 < \kappa_{\min} < \kappa < \kappa_{\max}$.
During the simulation of continuum mechanical problems using the Navier-Stokes equations, the diffusion operator \eqref{eq:diffusion_operator} appears multiple times.
For example, in the momentum equation $\kappa$ represents the viscosity and $u$ a velocity component, and in the energy equation $u$ is the temperature and $\kappa$ is the heat conductivity.
In pressure Poisson equations, $u$ represents the pressure and $\kappa$ the specific volume.
All of these representatives of $\kappa$ can exhibit jumps during phase change processes.
Most notably, the viscosity can jump by several orders of magnitude when modeling solid materials as extremely viscous fluids.
With this in view, it is necessary to study the case of strong discontinuities in the diffusion coefficient $\kappa$ that leads to an interface problem.
In this paper, we only consider the case of a linear diffusion operator with a diffusivity $\kappa=\kappa(\bvec{x})$ that only depends on the spatial variable $\bvec{x}$.

Classical mesh-based methods, such as the finite element method and the finite volume method, facilitate the discretization of diffusion operators as in \cref{eq:diffusion_operator} because they require the weak or integral form of partial differential equations \cite{Babuska_1970,Ewing_Li_Lin_Lin_1999}.
In contrast, the generalized finite difference method uses differential operators in their strong form with an MLS ansatz to enforce consistency conditions for the discrete operators \cite{Fan_Chu_Sarler_Li_2019,Li_Fan_2017,Suchde_2018}.
The MLS ansatz lacks stability features that are intrinsic to some mesh-based methods, such as diagonal dominance, which is essential for stability of the method and fulfillment of the discrete maximum principle for elliptic problems \cite{Seibold_2006,Chipot_2009}.
We present a one-dimensional correction technique to enforce diagonal dominance upon the discrete operators.
Furthermore, we investigate derived operators that stem from modifications applied to discrete differential operators.
This will be the foundation for the new discrete diffusion operator.

A common approach to tackle discontinuities in the diffusion coefficient is to divide the domain into subdomains on which the diffusivity is smooth.
Such domain decompositions require explicit interface conditions between the subdomains and points on the interface \cite{Xing_Song_He_Qiu_2020,Ahmad_Islam_Larsson_2020,Davydov_Safarpoor_2021,Qin_Song_Liu_2023}.
In the present work, we assume an unknown interface location rendering it impossible to perform a domain decomposition without prior interface identification.
Moreover, placing points on the interface is impossible in a diffuse interface scenario.
Yoon and Song \cite{Yoon_Song_2014}, Kim et al. \cite{Kim_Liu_Yoon_Belytschko_Lee_2007}, and Suchde and Kuhnert \cite{Suchde_Kuhnert_2019} incorporate enrichment to improve the discrete diffusion operator by adding non-differentiable functions to the consistency conditions in assumed interface regions.
The new discrete diffusion operator presented in this paper does not require a special treatment of interface points and hence knowledge of the interface location.
Moreover, the new operator inherits diagonal dominance from the discrete Laplace operator and thus yields a stable numerical scheme.

Following a brief introduction to generalized finite difference methods in \cref{sec:gfdm}, we formulate the new derived diffusion operator in \cref{sec:diffusion_operator} and study it regarding consistency and stability. We test the new discrete diffusion operator in \cref{sec:results} using the parabolic heat equation
\begin{equation} \label{eq:heat_equation_base}
    \diffp{u}{t} = \nabla\cdot(\kappa\nabla u) + Q
\end{equation}
and the elliptic Poisson's equation
\begin{equation} \label{eq:poisson_equation}
    -\nabla\cdot(\kappa\nabla u) = Q.
\end{equation}
It is possible to incorporate non-homogeneous Dirichlet and Neumann boundary conditions in the generalized finite difference method.
But in this paper, we only impose homogeneous Dirichlet boundary conditions \(u|_{\partial\Omega} = 0\) on the boundary of a domain $\Omega\subset\mathbb{R}^d$ in the two-dimensional case $d=2$.
The ideas presented in this work can be easily extended to the three-dimensional case $d=3$.

\section{Generalized finite difference method} \label{sec:gfdm}
Our formulation of the generalized finite difference method is based on the discretization of a closed domain $\Omega\subset\mathbb{R}^d$ by a point cloud \( \Omega_h = \set{\bvec{x}_1, \dots, \bvec{x}_N}\subset\Omega. \)
The indexed function $h \colon \Omega \to \R^+$ is called interaction radius or smoothing length and establishes a concept of connectivity through the discrete balls given by \( B_i = \set{\bvec{x} \in \Omega_h \mid \norm{\bvec{x} - \bvec{x}_i} \le h(\bvec{x}_i)} \) defining the stencils \(\S_i = \set{j \mid \bvec{x}_j \in B_i}. \)
Since a point with index $i$ is always a part of its stencil, $i\in\S_i$, it is convenient to define the neighbors $\N_i = \S_i\setminus\set{i}$.
Throughout the document, we use the Landau notation $O(h^p)$ without specifying the limit $h \to 0$.
This limit means that the interaction radius $h$ converges uniformly to the zero function.

\subsection{Differential operator discretization}
In generalized finite difference methods, linear differential operators $D$ are discretized in their strong form at each point $\bvec{x}_i\in\Omega_h$.
For this, we compute coefficients $c_{ij}^D$ to represent the discrete differential operator $D_i$ at $\bvec{x}_i$ by
\begin{equation} \label{eq:gfdm_formulation}
    Du(\bvec{x}_i) \approx D_i u = \sum_{j\in\S_i} c_{ij}^D u(\bvec{x}_j).
\end{equation}
For example, the coefficients $c_{ij}^\Delta$ represent the discrete Laplace operator $\Delta_i$ and the vector-valued coefficients $\bvec{c}_{ij}^\nabla\in\mathbb{R}^d$ express the discrete gradient $\nabla_i$.
The numerous ways to calculate the coefficients $c_{ij}^D$ divide generalized finite difference methods into distinct formulations, for example, RBF methods \cite{Larsson_Fornberg_2003}, RBF-FD methods \cite{Flyer_Fornberg_Bayona_Barnett_2016,Bayona_Moscoso_Carretero_Kindelan_2010,Shankar_2017}, or Voronoi-based finite volume methods (FVM) \cite{Mishev_1998}.
Milewski \cite{Milewski_2018} applies a random walk technique to obtain the coefficients and stencils, while Davydov and Safarpoor \cite{Davydov_Safarpoor_2021} select the stencils based on quality measures.
The latter two approaches generally lead to non-radial neighborhoods $B_i$.

For our MLS formulation, we define a test function set $\Phi_i$ to enforce exact reproducibility
\begin{equation} \label{eq:wlsq_reproducibility}
    D\phi(\bvec{x}_i) = \sum_{j\in\S_i} c_{ij}^D \phi(\bvec{x}_j),\quad\forall\phi\in\Phi_i.
\end{equation}
This leads to a linear system $\bmat{K}_i\bvec{c}_i^D = \bvec{b}_i^D$ where the consistency matrix $\bmat{K}_i$ contains the evaluations of the test functions at the neighboring points $\phi(\bvec{x}_j)$, $\bvec{c}_i^D$ is a collection of the respective coefficients $c_{ij}^D$, and $\bvec{b}_i^D$ is the right-hand side resulting from the evaluations of the test function derivatives $D\phi(\bvec{x}_i)$.
If the number of neighbors coincides with the number of test functions, $\abs{B_i} = \abs{\Phi_i}$, \cref{eq:wlsq_reproducibility} is solvable, provided that the resulting matrix $\bmat{K}_i$ is invertible.
But the MLS method originates from neighborhoods that have more points than there are test functions, $\abs{B_i} > \abs{\Phi_i}$.
To solve \cref{eq:wlsq_reproducibility} in this case, we impose a minimization
\begin{equation} \label{eq:wlsq_minimization}
    \min \sum_{j\in\S_i} \frac{1}{2} \left( \frac{c_{ij}^D}{w_{ij}} \right)^2
\end{equation}
with weights \( w_{ij} = w(\norm{\bvec{x}_j - \bvec{x}_i}/ h_i ) \) given by a decreasing weight function $w \colon [0, 1] \to (0, 1]$ that we set as $w(r) = \exp(-r)$.
Generally, the choice of the weight function influences the discretization error but is not a subject of this paper \cite{Jacquemin_Tomar_Agathos_Mohseni-Mofidi_Bordas_2020}.

Together with \cref{eq:wlsq_minimization}, we rewrite \cref{eq:wlsq_reproducibility} in matrix form
\begin{equation} \label{eq:wlsq_optimization}
    \min\Set{\frac{1}{2} \norm{\bmat{W}_i^{-1} \bvec{c}_i^D}_2^2 | \bmat{K}_i \bvec{c}_i^D = \bvec{b}_i^D}
\end{equation}
where $\bmat{W}_i$ is a diagonal matrix with the weights $w_{ij}$.
The optimization problem has the unique solution $\bvec{c}_i^D = \bmat{W}_i^2 \bmat{K}_i^T \bmat{A}_i^{-1} \bvec{b}_i^D$ if
\begin{equation} \label{eq:wlsq_optimization_matrix}
    \bmat{A}_i = \bmat{K}_i \bmat{W}_i^2 \bmat{K}_i^T
\end{equation}
is invertible which we will always assume in the following.
For that, it is necessary that the set of test functions $\Phi_i$ is unisolvent.
A more detailed discussion is presented by Wendland \cite{Wendland_2004}.
In our computations, it is usually sufficient that the neighboring points do not lie on a line (in 2D) or a plane (in 3D).

In the following, the multi-index notation with the conventions $\bvec{x}^{\bvec{\alpha}} = \prod_{i=1}^d x_i^{\alpha_i}$ and $\abs{\bvec{\alpha}} = \sum_{i=1}^d \alpha_i$ for $\bvec{\alpha} \in \mathbb{N}_0^d$ and $\bvec{x} \in \mathbb{R}^d$ is used.
We will further use the notations $\bvec{\alpha}! = \prod_{i=1}^d \alpha_i!$, and $\partial^{\bvec{\alpha}} = \prod_{i=1}^d \partial_i^{\alpha_i}$ to express Taylor expansions in a compact form.
Additionally, with another multi-index $\bvec{\beta}\in\mathbb{N}_0^d$, we will use the binomial coefficient $\binom{\bvec{\alpha}}{\bvec{\beta}} = \frac{\bvec{\alpha}!}{(\bvec{\alpha}-\bvec{\beta})! \bvec{\beta}!}$, and we say that $\bvec{\alpha}\le\bvec{\beta}$ if $\alpha_i \le \beta_i$ for each $i=1,\dots,d$.

We use monomials up to degree $p$
\begin{equation}
    \Phi_i = \set{\bvec{x}\mapsto(\bvec{x} - \bvec{x}_i)^{\bvec{\alpha}} | \bvec{\alpha}\in\mathbb{N}_0^d, \, \abs{\bvec{\alpha}} \le p}
\end{equation}
to enforce the consistency of the discrete differential operators.
If $\bmat{A}_i$ in \cref{eq:wlsq_optimization_matrix} is invertible, then its condition number depends on the smoothing length $h$ and tends to $\infty$ when $h \to 0$, however it becomes independent of the smoothing length for scaled monomials $(\bvec{x} - \bvec{x}_i)^{\bvec{\alpha}} / h(\bvec{x}_i)^{\abs{\bvec{\alpha}}}$ \cite{Zheng_Li_2022}.
For clarity, we will use the notation with unscaled monomials, while for the numerical investigations in \cref{sec:results} we use scaled monomials.

Monomial test functions enable us to discretize basic differential operators, such as the Laplace operator or directional derivatives.
To illustrate our formulation of the generalized finite difference method, let us consider the Laplace operator and monomial test functions.
The reproducibility conditions from \cref{eq:wlsq_reproducibility} read
\begin{equation} \label{eq:consistency_laplace}
    \sum_{j\in\S_i} c_{ij}^\Delta (\bvec{x}_j-\bvec{x}_i)^{\bvec{\alpha}} =
    \begin{cases}
        2, &\text{if } \bvec{\alpha} = 2\bvec{e}_k, \\
        0, &\text{else},
    \end{cases}
\end{equation}
where $\bvec{e}_k$ is the $k$th standard basis vector of $\mathbb{R}^d$.
According to these consistency conditions, the right-hand side $\bvec{b}_i^\Delta$ in the constraint in optimization problem \eqref{eq:wlsq_optimization} is independent of $\bvec{x}_i$.
On the other hand, by expanding the diffusion operator \eqref{eq:diffusion_operator} for a sufficiently smooth diffusivity $\kappa$
\begin{equation} \label{eq:diffusion_expanded}
    \nabla\cdot(\kappa\nabla u) = \dotprod{\nabla\kappa}{\nabla u} + \kappa\Delta u,
\end{equation}
we can derive the consistency conditions
\begin{equation} \label{eq:consistency_diffusion}
    \sum_{j\in\S_i} c_{ij}^{\nabla\cdot(\kappa\nabla)} (\bvec{x}_j-\bvec{x}_i)^{\bvec{\alpha}}
    =
    \begin{cases}
        \partial_k \kappa(\bvec{x}_i), &\text{if } \bvec{\alpha} = \bvec{e}_k, \\
        2\kappa(\bvec{x}_i), &\text{if } \bvec{\alpha} = 2\bvec{e}_k, \\
        0, &\text{else}.
    \end{cases}
\end{equation}
In contrast to $\bvec{b}_i^\Delta$, the right-hand side $\bvec{b}_i^{\nabla\cdot(\kappa\nabla)}$ depends on $\bvec{x}_i$ and needs the gradient $\nabla\kappa(\bvec{x}_i)$ or an approximation thereof for problems where the gradient cannot be computed manually.
For sufficiently smooth $\kappa$, the approximation of $\nabla\kappa$ is a reasonable approach to discretize the diffusion operator.
However, in our previous work we have shown that for discontinuous $\kappa$ this approach leads to instabilities for elliptic interface problems \cite{Kraus_Kuhnert_Meister_Suchde_2023}.
The new discrete diffusion operator presented in \cref{sec:diffusion_operator} circumvents the necessity of explicitly calculating the diffusivity gradient and the right-hand side for the optimization problem \eqref{eq:wlsq_optimization}.

Note that each discrete differential operator $D_i$ that differentiates constant functions exactly fulfills $c_{ii}^D = -\sum_{j\in\N_i} c_{ij}^D$. Because of that, the diagonal entry results from the off-diagonal entries and doesn't need to be computed.
Moreover, the application of a discrete differential operator to a function $u$ can be equivalently to \cref{eq:gfdm_formulation} written as
\begin{equation} \label{eq:gfdm_formulation_neighbors}
    D_i u = \sum_{j\in\N_i} c_{ij}^D (u(\bvec{x}_j) - u(\bvec{x}_i)).
\end{equation}

\subsection{Accuracy} \label{sec:accuracy}
In the previous section, we have derived an MLS approach to discretizing basic differential operators of degree $m$ of the form
\begin{equation} \label{eq:accuracy_operator}
    D = \sum_{\abs{\bvec{\alpha}}=0}^m a_{\bvec{\alpha}}\partial^{\bvec{\alpha}}
\end{equation}
with constant coefficients $a_{\bvec{\alpha}}\in\mathbb{R}$, and $a_{\bvec{\alpha}}\ne 0$ for some $\abs{\bvec{\alpha}}=m$ by using exact reproducibility of monomial test functions.
In this case, the coefficients for the respective discrete differential operator $D_i$ behave as $c_{ij}^D = \bigcap_{k=0}^m O(h^{-k})=O(h^{-m})$ for $h\to0$ due to the monomial consistency conditions.
This enables us to formulate the local discretization error for a sufficiently smooth function $u$.
A Taylor expansion of $u$ at each neighbor $\bvec{x}_j$ for $j\in\S_i$ yields
\begin{equation} \label{eq:accuracy_taylor}
    D_i u = \sum_{\abs{\bvec{\alpha}}=0}^p \frac{\partial^{\bvec{\alpha}}u(\bvec{x}_i)}{\bvec{\alpha}!} \sum_{j\in\S_i} c_{ij}^D (\bvec{x}_j - \bvec{x}_i)^{\bvec{\alpha}} + \sum_{j\in\S_i} c_{ij}^D O(h^{p+1}).
\end{equation}
Applying the consistency conditions of the differential operator \eqref{eq:accuracy_operator} to the inner sum in \cref{eq:accuracy_taylor} and the behavior of the coefficients in the limit $h\to0$ to the last sum, we obtain the local discretization error
\begin{equation} \label{eq:consistency_error}
    \tau_i = Du(\bvec{x}_i) - \sum_{j\in\S_i} c_{ij}^D u(\bvec{x}_j) = O(h^{p-m+1})
\end{equation}
for $h\to0$.
This means that the discrete differential operator $D_i$ is consistent if the degree of monomials is at least equal to the order of the differential operator $p\ge m$.
Hence, using second-degree monomials yields a second-order discrete gradient and a first-order discrete Laplace operator.
Using first-degree monomials to calculate the discrete Laplacian does not result in a consistent discretization.
This can also be seen in the consistency conditions in \cref{eq:consistency_laplace} because the principal nonzero terms on the right-hand side only appear for second-degree monomials.

\begin{remark} \label{rem:approximate_consistency_conditions}
    If the consistency conditions are not exactly reproduced, that is $\bmat{K}_i\bvec{c}_i^D = \bvec{b}_i^D + O(h^q)$ for some $q\in\mathbb{N}$, then the additional error term directly transfers to the consistency error in \cref{eq:consistency_error}.
    This can be easily verified by using the approximated consistency conditions in the Taylor expansion \eqref{eq:accuracy_taylor} which yields $\tau_i = O(h^{p-m+1}) + O(h^q)$.
\end{remark}

\subsection{Stability} \label{sec:stability}
So far we have only considered consistency conditions for deriving discrete differential operators.
However, it is well-known that a consistent numerical scheme is convergent if and only if it is stable.
The notion of stability usually depends on the mathematical model and gives sufficient or necessary stability conditions to the numerical scheme.
In this section we give a brief outline of stability with respect to solving Poisson's equation while the study of the semi-discretization of the heat equation yields similar conditions.
A discretization of Poisson's equation \eqref{eq:poisson_equation} with homogeneous Dirichlet boundary conditions results in a linear system $\bmat{A}\bvec{u} = \bvec{b}$ with a sparse matrix $\bmat{A}\in\mathbb{R}^{N \times N}$, the right-hand side resulting from the source term $\bvec{b}\in\mathbb{R}^N$ and a solution vector $\bvec{u}\in\mathbb{R}^N$.
In $\bmat{A}$, rows corresponding to interior points consist of the coefficients obtained from discretizing the diffusion operator $a_{ij} = -c_{ij}^{\nabla\cdot(\kappa\nabla)}$ for $j\in\S_i$, and all other rows belong to Dirichlet boundary points with $a_{ij}=\delta_{ij}$ with the Kronecker delta symbol $\delta_{ij}$.
Let us consider the vectorized analytical solution $\bvec{u} = (u(\bvec{x}_1), \dots, u(\bvec{x}_N))^T$ and a solution of the discretized Poisson's equation $\bvec{u}_h = (u_1, \dots, u_N)^T$ and define the global error by $\bvec{e}=\bvec{u}-\bvec{u}_h$.
Thus, the numerical scheme is convergent if $\norm{\bvec{e}}\to0$ holds in some norm $\norm{\cdot}$ as $h\to0$.
The global error can also be expressed with the local discretization errors as $\bmat{A}\bvec{e}=\bvec{\tau}$ with $\bvec{\tau} = (\tau_1,\dots,\tau_N)^T$.
We finally obtain $\norm{\bvec{e}} \le \norm{\bmat{A}^{-1}}\cdot\norm{\bvec{\tau}}$ if $\bmat{A}$ is invertible.
Since a consistent numerical scheme fulfills $\norm{\bvec{\tau}}\to0$ for $h\to0$, a sufficient condition for the stability is the boundedness of $\norm{\bmat{A}^{-1}}$ independent of the discretization size $h$ and thus number of points $N$.

In the present generalized finite difference method, the sole information about the system matrix $\bmat{A}$ are row properties due to the consistency conditions of the discrete diffusion operator.
Because of that, it is reasonable to look for a bound of $\bmat{A}$ in the matrix norm induced by the maximum norm
\begin{equation}
    \norm{\bmat{A}}_\infty = \max_{i=1,\dots,N} \sum_{j\in\S_i} \abs{a_{ij}}.
\end{equation}
Collatz \cite{Collatz_1952}, and Bramble and Hubbard \cite{Bramble_Hubbard_1962} found in their studies that the system matrix resulting from a discretization of Poisson's equation is bounded in the maximum norm if it is monotone, that means the inverse only consists of non-negative entries $\bmat{A}^{-1}\ge0$.
A sufficient condition for the monotony is diagonal dominance of the matrix $\bmat{A}$ with positive diagonal entries $a_{ii} > 0$ and non-positive off-diagonal entries $a_{ij} \le 0$ leading to an M-matrix \cite{Seibold_2006}.
Lorenz \cite{Lorenz_1977} presented a condition for a matrix being a product of finitely many M-matrices by relaxing the diagonal dominance condition, and allowing small positive off-diagonal entries compared to the diagonal entry.
Diagonal dominance is also crucial for solving the parabolic heat equation, where the lack of diagonal dominance can lead to severe instabilities \cite{Suchde_2018}.

Monotone matrices do not only guarantee stability of the scheme, but also fulfill discrete maximum principles.
The weak maximum principle for Poisson's equation with homogeneous boundary conditions states that if the source term is non-positive $Q\le 0$, then the analytical solution attains its maximum at the boundary such that in our case $u\le 0$ is fulfilled.
Similarly, the right-hand side of the discrete problem satisfies $\bvec{b}\le 0$ component-wise for $Q\le 0$, and if $\bmat{A}$ is monotone then the solution of the linear system satisfies $\bvec{u}_h = \bmat{A}^{-1}\bvec{b} \le 0$.
This concept can also be extended to non-homogeneous Dirichlet boundary conditions but is not within the scope of this paper.
A more detailed discussion is provided by Ciarlet \cite{Ciarlet_1970} in the context of the finite difference method, but the statements apply similarly to the generalized finite difference method.

\subsection{Diagonal dominance} \label{sec:diagonal_dominance}
Since diagonally dominant operators pose a sufficient condition to the stability of the numerical scheme, it is necessary to construct diagonally dominant discrete operators.
Motivated by the findings from the previous section, we define diagonally dominant discrete operators as follows.
\begin{definition} \label{def:diagonal_dominance}
    A discrete differential operator $D_i$ is called diagonally dominant if it is consistent for constant functions, and its coefficients $c_{ij}^D$ satisfy the sign condition $c_{ij}^D c_{ii}^D \le 0$ for all $j \in\N_i$.
\end{definition}
Because the optimization problem \eqref{eq:wlsq_optimization} does not enforce diagonal dominance, a correction technique was proposed by Suchde \cite{Suchde_2018}.
For the correction, we use a second discrete operator $\bvec{c}_i^0\in\ker{\bmat{K}_i}$ approximating the zero functional $f \mapsto 0$ non-trivially by prescribing $c_{ii}^0 = 1$ in the optimization.
The corrected operator $\widehat{\bvec{c}}_i^D = \bvec{c}_i^D + \widehat{\alpha} \bvec{c}_i^0
$ thus satisfies the same consistency conditions as $\bvec{c}_i^D$ and hence does not reduce the consistency and convergence order.
One way to obtain the coefficient $\widehat{\alpha}$ is to formulate a minimization problem that penalizes large off-diagonal entries relative to the diagonal entry
\begin{equation} \label{eq:alpha_minimization}
    \widehat{\alpha} = \operatorname*{arg\,min}_{\alpha\in\mathbb{R}} \sum_{j\in\S_i} \frac{(c_{ij}^D + \alpha c_{ij}^0)^2}{(c_{ii}^D + \alpha c_{ii}^0)^2}.
\end{equation}
The advantage of this minimization approach is its unique solvability and the low computational cost. But a disadvantage is that the sign condition from \cref{def:diagonal_dominance} is not necessarily fulfilled.
In our previous work \cite{Kraus_Kuhnert_Meister_Suchde_2023}, we have demonstrated the applicability of this correction technique for obtaining diagonally dominant discrete operators.
However, we have also exemplarily demonstrated that stencils of sufficiently high quality are required without specifying a rigorous definition of point cloud quality.
We could also show that diagonally dominant operators could only be calculated with this method in the case of a sufficiently smooth diffusion coefficient, and no diagonally dominant operators were obtained in the case of a discontinuous diffusion coefficient.

\subsection{Derived operators} \label{sec:derived_operators}
In this section, we lay the foundation for the discretization of the variable coefficient diffusion operator in \cref{sec:diffusion_operator} with the use of so-called \emph{derived operators}.
By derived operators we understand discrete operators that stem from weighting coefficients of another discrete operator.
Let $c_{ij}^D$ be the coefficients of a discrete differential operator approximating $D$, and $v_j\in\mathbb{R}$ some weights, then the derived operator reads $c_{ij}^{\widetilde{D}} = v_j c_{ij}^{D}$ and discretizes another operator $\widetilde{D}$.
If the weights $v_j$ stem from discrete evaluations of a function at neighboring points, that means $v_j = v(\bvec{x}_j)$, then derived operators can be studied using the product rule.
Let $u$, $v\colon\Omega\to\mathbb{R}$ be two sufficiently smooth functions, then by the Leibniz-rule
\begin{equation} \label{eq:leibniz-rule}
    \partial^{\bvec{\alpha}} (uv) = \sum_{\bvec{\beta}\le\bvec{\alpha}} \binom{\bvec{\alpha}}{\bvec{\beta}} \partial^{\bvec{\alpha}-\bvec{\beta}} u \, \partial^{\bvec{\beta}} v
\end{equation}
holds for any multi-index $\bvec{\alpha}\in\mathbb{N}_0^d$.
For now, let $v(\bvec{x}) = (\bvec{x}-\bvec{x}_i)^{\bvec{\gamma}}$ be a monomial with a multi-index $\bvec{\gamma}\in\mathbb{N}_0^d$ and $\bvec{\gamma}\le\bvec{\alpha}$.

Applying the Leibniz-rule \eqref{eq:leibniz-rule} and evaluating the derivatives at $\bvec{x}=\bvec{x}_i$ yields
\begin{equation} \label{eq:leibniz-rule-applied}
    \partial^{\bvec{\alpha}} (uv) |_i = \frac{\bvec{\alpha}!}{(\bvec{\alpha}-\bvec{\gamma})!} \partial^{\bvec{\alpha}-\bvec{\gamma}} u(\bvec{x}_i)
\end{equation}
where we have used the property $\partial^{\bvec{\gamma}} (\bvec{x}-\bvec{x}_i)^{\bvec{\gamma}} = \bvec{\gamma}!$.
Let $D_i$ be a discrete operator that discretizes $D=\partial^{\bvec{\alpha}}$ and has a consistency error of $O(h^p)$ for some $p\in\mathbb{N}$.
Discretizing the term on the left-hand side of \cref{eq:leibniz-rule-applied} yields
\begin{equation} \label{eq:derived-operator-applied}
    \partial^{\bvec{\alpha}} (uv) |_i = \sum_{j\in\S_i} c_{ij}^D (\bvec{x}_j-\bvec{x}_i)^{\bvec{\gamma}} u(\bvec{x}_j) + O(h^p).
\end{equation}
Combining \cref{eq:leibniz-rule-applied,eq:derived-operator-applied}, we can see that the coefficients $c_{ij}^{\widetilde{D}} = c_{ij}^D (\bvec{x}_j - \bvec{x}_i)^{\bvec{\gamma}}$ pose a consistent discretization of $\widetilde{D} = \frac{\bvec{\alpha}!}{(\bvec{\alpha}-\bvec{\gamma})!} \partial^{\bvec{\alpha}-\bvec{\gamma}}$ with the same consistency order $p$.
Recall from \cref{sec:accuracy} that the consistency order depends on both the degree of monomials and the degree of the differential operator.
Hence, by enforcing the same consistency conditions to calculate a discrete differential operator of $\widetilde{D}$, the MLS ansatz could produce an approximation of higher order in the case of $\max\bvec{\alpha}-\bvec{\gamma} < \max\bvec{\alpha}$.
In the case of $\bvec{\gamma}>\bvec{\alpha}$, the resulting derived operator poses a discretization of the zero functional of order $p$.

We consider operators that can be derived from the discrete Laplace operator to define the new discrete diffusion operator in the following section.
Hence, we assume that the discrete Laplace operator satisfies
\begin{equation} \label{eq:laplace_accuracy}
    \Delta u(\bvec{x}_i) = \sum_{j\in\S_i} c_{ij}^\Delta u(\bvec{x}_j) + O(h^p)
\end{equation}
for some $p\in\mathbb{N}$.
In the following corollary we summarize all discrete operators that can be derived from the discrete Laplace operator by weighting with monomials.

\begin{corollary} \label{thm:derived_from_laplace}
    Let $u\colon\Omega\to\mathbb{R}$ be a sufficiently smooth function, and the coefficients $c_{ij}^\Delta$ fulfill \cref{eq:laplace_accuracy} then
    \begin{equation*}
        \sum_{j\in\S_i} c_{ij}^\Delta (\bvec{x}_j - \bvec{x}_i)^{\bvec{\alpha}} u(\bvec{x}_j) =
        \left.
        \begin{cases}
            \Delta u(\bvec{x}_i),        &\text{if } \bvec{\alpha} =  \bvec{0}, \\
            2\partial_{k} u(\bvec{x}_i), &\text{if } \bvec{\alpha} =  \bvec{e}_k, \\
            2u(\bvec{x}_i),              &\text{if } \bvec{\alpha} = 2\bvec{e}_k, \\
            0,                           &\text{else},
        \end{cases}
        \right\}
        + O(h^p)
    \end{equation*}
    holds for all $\bvec{\alpha}\in\mathbb{N}_0^d$.
\end{corollary}

\Cref{thm:derived_from_laplace} states that the coefficients $c_{ij}^\Delta(\bvec{x}_j-\bvec{x}_i)/2\in\mathbb{R}^d$ depict a derived gradient operator of order $p$.
Moreover, we obtain derived interpolation operators with $c_{ij}^\Delta(\bvec{x}_j - \bvec{x}_i)^{2\bvec{e}_k}/2$ for each $k\in\set{1,\dots,d}$ that can be used to interpolate function values at new nodes.
In this form, the derived interpolation operators only consider the variation of the nodes in the $k$th coordinate direction.
Arithmetic averaging of all derived interpolation operators yields an interpolation operator that is independent of the coordinate direction.
Note that all derived operators obtained from multiplication with a monomial with $\bvec{\alpha} \ne \bvec{0}$ have vanishing diagonal entries.

\section{Derived diffusion operator} \label{sec:diffusion_operator}
The new discrete diffusion operator is inspired by the Voronoi-based finite volume method, generalizing it to arbitrary neighborhoods that are not based on a mesh. Trask et al. \cite{Trask_Perego_Bochev_2017} presented a similar idea that requires a globally computed graph connecting the point cloud. Seifarth \cite{Seifarth_2018} also successfully applied ideas from classical finite volume methods to generalized finite difference methods to solve transport equations on static point clouds. Kwan-Yu et al. \cite{Kwan-yu_Chiu_Wang_Hu_Jameson_2012} used ideas from mesh-based methods to construct conservative differential operators. Their discrete differential operators, however, are computed globally, while we compute our differential operators locally.

We demonstrated in our previous work that the Voronoi-based finite volume method can be formulated as a generalized finite difference method \cite{Kraus_Kuhnert_Meister_Suchde_2023}.
This formulation automatically satisfies properties such as the discrete Gauss theorem and diagonal dominance.
To obtain the Voronoi-based finite volume method, we define the Voronoi cell surrounding a point $\bvec{x}_i$
\begin{equation} \label{eq:voronoi_cell}
    \Omega_i = \set{\bvec{x} \in \Omega \mid \norm{\bvec{x} - \bvec{x}_i} < \norm{\bvec{x} - \bvec{x}_j} \text{ for all } j\ne i}.
\end{equation}
In the finite volume method, a function $u$ is approximated by cell averages
\begin{equation} \label{eq:cell_average}
    u(\bvec{x}_i) \approx \frac{1}{\abs{\Omega_i}} \int_{\Omega_i} u(\bvec{x}) \dl V.
\end{equation}
This is a second-order approximation only at the centroid of the Voronoi cell $\Omega_i$ and not at $\bvec{x}_i$.
To discretize a diffusion operator, we split the boundary of $\Omega_i$ into line segments (in 2D) or surfaces (in 3D) $\Gamma_{ij} = \overline{\Omega}_i \cap \overline{\Omega}_j$ that lie between the cells corresponding to the points $\bvec{x}_i$ and $\bvec{x}_j$ respectively.
Applying the cell average ansatz \eqref{eq:cell_average} on $\Delta u$, Gauss's theorem and quadrature rules yield the discrete Laplace operator
\begin{equation*}
    \Delta u(\bvec{x}_i)
    \approx \sum_{j \in\N_i} \frac{1}{\abs{\Omega_i}} \int_{\Gamma_{ij}} \diffp{u}{\bvec{n}} \dl S
    \approx \sum_{j\in\N_i} \frac{\abs{\Gamma_{ij}}}{\abs{\Omega_i}} \frac{u(\bvec{x}_j) - u(\bvec{x}_i)}{\norm{\bvec{x}_j - \bvec{x}_i}}.
\end{equation*}
In the above equation, we used a central finite difference discretization to approximate the directional derivative at the mid-point $\bvec{x}_{ij} = \frac{\bvec{x}_i + \bvec{x}_j}{2}$ between $\bvec{x}_i$ and $\bvec{x}_j$.
With the coefficients $f_{ij}^{\Delta} = \frac{\abs{\Gamma_{ij}}}{\abs{\Omega_i}\norm{\bvec{x}_j - \bvec{x}_i}}$ for $j\in\N_i$, we obtain the generalized finite difference formulation as in \cref{eq:gfdm_formulation_neighbors}
\begin{equation} \label{eq:laplace_fvm}
    \Delta u(\bvec{x}_i) \approx \sum_{j\in\N_i} f_{ij}^{\Delta}(u(\bvec{x}_j) - u(\bvec{x}_i)).
\end{equation}
The coefficients $f_{ij}^\Delta$ describe the flux from $\Omega_i$ to $\Omega_j$ and are based on the geometric properties of the Voronoi cell.
Following similar steps, we can establish a discrete diffusion operator
\begin{equation} \label{eq:diffusion_fvm}
    \nabla\cdot(\kappa\nabla u)(\bvec{x}_i) \approx \sum_{j\in\N_i} \kappa_{ij} f_{ij}^{\Delta} (u(\bvec{x}_j) - u(\bvec{x}_i))
\end{equation}
where approximations $\kappa_{ij} \approx \kappa(\bvec{x}_{ij})$ at $\bvec{x}_{ij}\in\Gamma_{ij}$ are added in front of the coefficients $f_{ij}^\Delta$.
The coefficients of the Voronoi-based diffusion operator are finally calculated as $f_{ij}^{\nabla\cdot(\kappa\nabla)} = \kappa_{ij}f_{ij}^\Delta$ for $j\in\N_i$.
Hence, the Voronoi-based finite volume method operators represent a derived diffusion operator by weighting a discrete Laplace operator.

Motivated by this, we replace the coefficients $f_{ij}^\Delta$ with generalized finite difference coefficients $c_{ij}^\Delta$ to define the new coefficients $c_{ij}^{\nabla\cdot(\kappa\nabla)} = \kappa_{ij} c_{ij}^\Delta$ for $j\in\N_i$ leading to the derived diffusion operator
\begin{equation} \label{eq:discrete_diffusion_operator}
    \nabla\cdot(\kappa\nabla u)(\bvec{x}_i) \approx \sum_{j\in\N_i} \kappa_{ij} c_{ij}^{\Delta} (u(\bvec{x}_j) - u(\bvec{x}_i)).
\end{equation}
Recall that the diagonal entry $c_{ii}^{\nabla\cdot(\kappa\nabla)}$ follows directly as the negative sum of the off-diagonal entries.
In the above, $\kappa_{ij}$ are, similar to the finite volume approach, reconstructions of the diffusivity $\kappa$ at the mid-point $\bvec{x}_{ij} = \frac{\bvec{x}_i + \bvec{x}_j}{2}$.
The derived diffusion operator also generalizes the finite volume method, reinterpreting the coefficients $c_{ij}^\Delta$ as virtual fluxes with an underlying, hidden notion of meshless volumes and surfaces.
Note that the derived diffusion operator in \cref{eq:discrete_diffusion_operator} does not necessarily need an explicit calculation of $\nabla\kappa$ as opposed to the MLS approach.

As we pointed out before, diagonal dominance guarantees the stability of the numerical method and the discrete maximum principle for elliptic problems.
The following corollary formulates sufficient conditions for diagonally dominant derived diffusion operators.
\begin{corollary} \label{cor:diagonal_dominance}
    The derived diffusion operator in \cref{eq:discrete_diffusion_operator} at point $\bvec{x}_i$ is diagonally dominant if $c_{ij}^\Delta \ge 0$ and $\kappa_{ij} > 0$ for each neighbor $j\in\N_i$.
\end{corollary}
We always assume that $\kappa > 0$ is a positive function, such that it is natural to assume positive reconstructions $\kappa_{ij}$.
Calculating diffusion operators with the derived approach as opposed to the MLS ansatz consequently only requires a diagonally dominant discrete Laplace operator and positive reconstructions.
For the remainder of this section, we will present some ways for calculating the reconstructions $\kappa_{ij}$ with so-called \emph{reconstruction functions}, study the consistency conditions, and briefly outline the possibility of extending the idea of derived operators to general anisotropic diffusion operators.

\subsection{Reconstruction functions} \label{sec:reconstruction_functions}
First, let us investigate the reconstructions $\kappa_{ij}$.
We define reconstruction functions as functions that approximate the diffusivity at the mid-point between two points as follows.
\begin{definition} \label{def:reconstruction_function}
    Let $\kappa\colon\Omega\to\mathbb{R}$ be a function.
    A function $\kappa_{i\to}\colon\Omega\to\mathbb{R}$ is called a reconstruction function of $\kappa$ of order $q+1$ at $\bvec{x}_i\in\Omega_h$ if it satisfies
    \begin{equation*}
        \kappa_{i\to}(\bvec{x}_j) = \kappa\left(\frac{\bvec{x}_i+\bvec{x}_j}{2}\right) + O(h^{p+1})
    \end{equation*}
    for each $j\in\S_i$.
\end{definition}
With the help of reconstruction functions, we define the reconstructions as $\kappa_{ij} = \kappa_{i\to}(\bvec{x}_j)$ for each point $j\in\S_i$.
We call reconstructions symmetric if $\kappa_{ij} = \kappa_{ji}$ is always fulfilled, and otherwise we call the reconstructions asymmetric.
\Cref{lem:reconstruction_function} presents a condition for reconstruction functions in the case of a sufficiently smooth diffusivity $\kappa$ and reconstruction function $\kappa_{i\to}$.

\begin{lemma} \label{lem:reconstruction_function}
    Let $\kappa$, $\kappa_{i\to}\in C^q(\Omega)$ be sufficiently smooth.
    Then $\kappa_{i\to}$ is a reconstruction function of $\kappa$ of order $q+1$ at point $\bvec{x}_i$ if and only if
    \begin{equation} \label{lemeq:reconstruction_function}
        \partial^{\bvec{\alpha}} \kappa_{i\to}(\bvec{x}_i) =
        \frac{\partial^{\bvec{\alpha}}\kappa(\bvec{x}_i)}{2^{\abs{\bvec{\alpha}}}} + O(h^{q+1-\abs{\bvec{\alpha}}})
    \end{equation}
    holds for all $\bvec{\alpha}\in\mathbb{N}_0^d$ with $\abs{\bvec{\alpha}}\le q$.
\end{lemma}

\begin{proof}
    A Taylor expansion at $\bvec{x}_i$ yields after some computational steps
    \begin{equation*}
        \kappa_{i\to}(\bvec{x}_j) - \kappa\left(\frac{\bvec{x}_i+\bvec{x}_j}{2}\right)
        =
        \sum_{\abs{\bvec{\alpha}}=0}^q
        \left[
            \frac{\partial^{\bvec{\alpha}}\kappa_{i\to}(\bvec{x}_i)}{\bvec{\alpha}!}
            -
            \frac{\partial^{\bvec{\alpha}}\kappa(\bvec{x}_i)}{2^{\abs{\bvec{\alpha}}}\bvec{\alpha}!}
        \right]
        (\bvec{x}_j-\bvec{x}_i)^{\bvec{\alpha}}
        + O(h^{q+1})
    \end{equation*}
    for each $j\in\N_i$.
    Hence, \cref{lemeq:reconstruction_function} provides a sufficient and necessary condition for the reconstruction property according to \cref{def:reconstruction_function}.
\end{proof}

In the following, let us investigate some possible reconstruction functions based on the discrete values $\kappa_i=\kappa(\bvec{x}_i)$ for $i=1,\dots,N$.
We will start by gradient-free reconstruction functions that do not require the calculation of $\nabla\kappa$.
The three Pythagorean means
\begin{subequations} \label{eq:pythagorean_means}
    \begin{align}
        \kappa_{ij} &= \frac{\kappa_i + \kappa_j}{2},                 \label{eq:arithmetic_mean} \\
        \kappa_{ij} &= 2\frac{\kappa_i\kappa_j}{\kappa_i + \kappa_j}, \label{eq:harmonic_mean} \\
        \kappa_{ij} &= \sqrt{\kappa_i \kappa_j},                      \label{eq:geometric_mean}
    \end{align}
\end{subequations}
which are the arithmetic, harmonic, and geometric mean, respectively, satisfy the reconstruction properties from \cref{lem:reconstruction_function}.
Reconstructions based on Pythagorean means provide symmetric reconstructions $\kappa_{ij} = \kappa_{ji}$ and are second-order accurate.
They also provide positive reconstructions $\kappa_{ij} > 0$ provided $\kappa_i$, $\kappa_j > 0$ is satisfied.

Numerically computing the gradients $\nabla_i\kappa$ for all points $i=1,\dots,N$ and using a Taylor expansion
\begin{subequations} \label{eq:taylor_reconstructions}
\begin{equation} \label{eq:taylor}
    \kappa_{ij} = \kappa_i + \frac{1}{2}\dotprod{\nabla_i\kappa}{\bvec{x}_j - \bvec{x}_i},
\end{equation}
or a skew Taylor expansion
\begin{equation} \label{eq:skew_taylor}
    \kappa_{ij} = \kappa_j - \frac{1}{2}\dotprod{\nabla_i\kappa}{\bvec{x}_j - \bvec{x}_i}
\end{equation}
\end{subequations}
yields second-order reconstructions.
Empirically, these Taylor-based reconstructions lead to worse results than the Pythagorean means, possibly due to the asymmetry of the reconstructions $\kappa_{ij} \ne \kappa_{ji}$, while requiring additional computational steps.
In addition, positive reconstructions cannot be guaranteed such that a correction of the discrete gradients $\nabla_i \kappa$ or the use of limiters has to be applied.

A higher-order scheme follows from a one-dimensional Hermite interpolation
\begin{equation} \label{eq:hermite_interpolation}
    \kappa_{ij} = \frac{\kappa_i + \kappa_j}{2} + \frac{1}{8} \dotprod{\nabla_i \kappa - \nabla_j \kappa}{\bvec{x}_j - \bvec{x}_i}.
\end{equation}
In this case, the reconstructions are symmetric and do not require more computation time than the Taylor-based reconstructions in \cref{eq:taylor_reconstructions}.
Similarly to the Taylor-based reconstructions, the reconstructions are not necessarily positive.

\subsection{Consistency conditions}
Restricting ourselves to reconstruction functions that fulfill the property from \cref{lem:reconstruction_function}, we derive the consistency conditions for the derived discrete diffusion operator.
\Cref{thm:diffusion_consistency_conditions} shows that the discrete diffusion operator meets the consistency conditions from \cref{eq:consistency_diffusion} approximately, as opposed to the exact reproducibility property of the discrete Laplacian.
It also shows that derived diffusion operators have an intrinsic notion of the gradient of the diffusivity, and offer a regularization at the scale of the MLS discretization.

\begin{theorem} \label{thm:diffusion_consistency_conditions}
    Let $\kappa$, $\kappa_{i\to} \in C^q(\Omega)$ and the reconstruction functions $\kappa_{i\to}$ fulfill the reconstruction property in \cref{lem:reconstruction_function}.
    If the coefficients $c_{ij}^\Delta$ fulfill \cref{eq:laplace_accuracy}, then the derived diffusion operator satisfies the consistency conditions
    \[ \sum_{j\in\S_i} c_{ij}^{\nabla\cdot(\kappa\nabla)}(\bvec{x}_j - \bvec{x}_i)^{\bvec{\alpha}} =
    \begin{cases}
        0,                                               &\text{if } \bvec{\alpha} = \bvec{0},  \\
        \partial_k \kappa(\bvec{x}_i) + O(h^p) + O(h^q), &\text{if } \bvec{\alpha} = \bvec{e}_k, \\
        2\kappa(\bvec{x}_i) + O(h^p) + O(h^{q+1}),       &\text{if } \bvec{\alpha} = 2\bvec{e}_k, \\
        O(h^p),                                          &\text{else},
    \end{cases} \]
    for all $\bvec{\alpha}\in\mathbb{N}_0^d$.
\end{theorem}

\begin{proof}
    The case for $\bvec{\alpha}=\bvec{0}$ follows directly from the definition of the coefficients $c_{ij}^{\nabla\cdot(\kappa\nabla)}$.
    For $\abs{\bvec{\alpha}} > 0$, we apply \cref{thm:derived_from_laplace} with the reconstruction function $\kappa_{i\to}$ in place of $u$ to obtain
    \begin{equation*}
        \sum_{j\in\S_i} c_{ij}^{\nabla\cdot(\kappa\nabla)} (\bvec{x}_j - \bvec{x}_i)^{\bvec{\alpha}} =
        \left.
        \begin{cases}
            2\partial_k \kappa_{i\to}(\bvec{x}_i),  &\text{if } \bvec{\alpha} = \bvec{e}_k, \\
            2\kappa_{i\to}(\bvec{x}_i),             &\text{if } \bvec{\alpha} = 2\bvec{e}_k,\\
            0,                                      &\text{else},
        \end{cases}
        \right\}
        + O(h^p).
    \end{equation*}
    Applying \cref{lem:reconstruction_function} yields
    \begin{align*}
        2\partial_k \kappa_{i\to}(\bvec{x}_i) &= \partial_k \kappa(\bvec{x}_i) + O(h^{q}), \\
        2\kappa_{i\to}(\bvec{x}_i)            &= 2\kappa(\bvec{x}_i) + O(h^{q+1}),
    \end{align*}
    concluding the proof.
\end{proof}

\Cref{thm:diffusion_consistency_conditions} allows us to derive the accuracy of the discrete diffusion operator.
As stated in \cref{thm:diffusion_accuracy}, the new discrete diffusion operator inherits the accuracy of the underlying discrete Laplace operator with an additional error term that depends on the accuracy of the reconstruction function.
We will record this statement in the following corollary.

\begin{corollary} \label{thm:diffusion_accuracy}
    Let $u \in C^{p+2}(\Omega)$, $\kappa$, $\kappa_{i\to}\in C^{q}(\Omega)$, and the reconstruction function $\kappa_{i\to}$ fulfill the reconstruction property in \cref{lem:reconstruction_function}.
    If the coefficients $c_{ij}^\Delta$ meet the condition in \cref{eq:laplace_accuracy}, then the derived diffusion operator fulfills
    \[ \nabla\cdot(\kappa\nabla u)(\bvec{x}_i) = \sum_{j\in\S_i} c_{ij}^{\nabla\cdot(\kappa\nabla)} u(\bvec{x}_j) + O(h^p) + O(h^q). \]
\end{corollary}

\begin{proof}
    The calculation of the accuracy of the derived diffusion operator follows the same steps as shown in \cref{sec:accuracy}.
    Only the different error terms from \cref{thm:diffusion_consistency_conditions} need to be accounted for by using $O(h^{q+1}) = O(h^q)$ for $h\to0$.
    Application of \cref{rem:approximate_consistency_conditions} finalizes the proof.
\end{proof}

\subsection{Anisotropic diffusion}
Recall that the most general case of diffusion operators is based on a general diffusion matrix $(\kappa^{kl})_{k,l=1}^d=\bmat{\kappa}$ as in \cref{eq:diffusion_operator}.
But derived operators can also be used to define discrete anisotropic diffusion operators in the case of a symmetric diffusivity $\bmat{\kappa}=\bmat{\kappa}^T$.
In that case, it is required to calculate all second-degree derivatives corresponding to non-zero entries in the diffusivity matrix $\bmat{\kappa}$, for example $c_{ij}^{x_1^2}$, $c_{ij}^{x_1 x_2}$, and $c_{ij}^{x_2^2}$ in 2D.
Let $c_{ij}^{kl}$ be the coefficients that discretize the differential operator $\difsp{}{x_k,x_l}$, then a derived anisotropic diffusion operator can be obtained by setting the coefficients
\begin{equation*}
    c_{ij}^{\nabla\cdot(\bmat{\kappa}\nabla)} = \sum_{k, l=1}^d \kappa_{ij}^{kl} c_{ij}^{kl}
\end{equation*}
for $j\in\N_i$.
As in the previous section, the weights $\kappa_{ij}^{kl} \approx \kappa^{kl}(\bvec{x}_{ij})$ are reconstructions of the individual components in the diffusivity matrix at the mid-points $\bvec{x}_{ij}=\frac{1}{2}(\bvec{x}_i+\bvec{x}_j)$.
From this, it is straightforward to verify that this discrete anisotropic diffusion operator fulfills the expected consistency conditions.

\subsection{An alternative view}
In \cref{sec:derived_operators}, we motivated derived operators by applying the Leibniz rule to the product of two functions.
Similarly, we can also obtain the new discrete diffusion operator by representing the diffusion operator in terms of other, well-known, operators and discretizing them instead.
For example, consider the formulation of the diffusion operator as $\nabla\cdot(\kappa\nabla u) = \dotprod{\nabla\kappa}{\nabla u} + \kappa\Delta u$ from \cref{eq:diffusion_expanded}.
Assuming that we have an approximation of the gradient $\nabla_i\kappa\approx\nabla\kappa(\bvec{x}_i)$, individual discretization of the remaining differential operators on the right-hand side yields a discrete diffusion operator.
Using $c_{ij}^\Delta$ to discretize the Laplacian and the derived gradient $c_{ij}^\Delta (\bvec{x}_j-\bvec{x}_i)/2$ to discretize the gradient, we obtain together with \cref{eq:gfdm_formulation_neighbors}
\begin{equation*}
    \nabla\cdot(\kappa\nabla u)(\bvec{x}_i) =
    \sum_{j\in\N_i} \left(\kappa(\bvec{x}_i)+\frac{1}{2}\dotprod{\nabla_i\kappa}{\bvec{x}_j-\bvec{x}_i}\right) c_{ij}^\Delta (u(\bvec{x}_j)-u(\bvec{x}_i)).
\end{equation*}
The coefficients represent a derived diffusion operator with Taylor-based reconstructions \eqref{eq:taylor}.
With similar steps, a derived diffusion operator with skew Taylor-based reconstructions \eqref{eq:skew_taylor} can be obtained from the reformulation $\nabla\cdot(\kappa\nabla u) = \Delta(\kappa u) - \nabla\kappa\cdot\nabla u - u\Delta\kappa$.
Arithmetic averaging of the prior reformulations yields $\nabla\cdot(\kappa\nabla u) = \frac{1}{2} (\Delta(\kappa u) + \kappa\Delta u - u\Delta\kappa)$ resulting in the arithmetic mean reconstructions \eqref{eq:arithmetic_mean}.
Currently, it is not known if there exist reformulations that result in the other reconstructions, for example the harmonic mean \eqref{eq:harmonic_mean} or Hermite interpolation \eqref{eq:hermite_interpolation}.

\section{Numerical results} \label{sec:results}
In this section, we test and compare the presented methods on a set of selected test cases.
But first, we introduce abbreviations for the methods.
The MLS based calculation of the diffusion operator with monomials up to degree $p$ is represented by MLS($p$, $q$) where we approximate the gradient with an MLS based discrete gradient of order $q$ by $\nabla_i\kappa = \nabla\kappa(\bvec{x}_i) + O(h^q)$.
The derived diffusion operator is characterized by the reconstruction function and the degree of monomials that is used to calculate the underlying discrete Laplace operator.
Hence it is abbreviated by DDO($p$, AVG) where $p$ is the degree of monomials for the discrete Laplace operator and AVG is an acronym for the reconstruction function.
Particularly, we use AM for the arithmetic mean \eqref{eq:arithmetic_mean}, HM for the harmonic mean \eqref{eq:harmonic_mean}, TE for the Taylor expansion \eqref{eq:taylor}, and HI for the Hermite interpolation \eqref{eq:hermite_interpolation}.
The remaining reconstruction functions presented in \cref{sec:reconstruction_functions} are not considered in this work because they did not produce significantly different results from those presented in the following.
Note that second-order discrete gradients have been used in the calculation of the Taylor expansion and the Hermite interpolation reconstructions.

All simulations are performed on unstructured two-dimensional point clouds with different refinement levels.
The point clouds were generated with the commercial software MESHFREE \cite{MESHFREE} with an advancing front point cloud generation technique \cite{Suchde_Jacquemin_Davydov_2023}, which enables us to demonstrate the applicability of our method for point clouds used in complex industrial applications.
Because the point clouds do not conform to any interface, we illustrate that our method does not require a special point cloud treatment.

We calculate the relative discrete $L^\infty$ error of the numerical solution $u_h$ given by the discrete values $u_i$ to the analytic reference solution $u\not\equiv 0$ as
\begin{equation*}
    \norm{u - u_h}_\infty = \frac{\max_{i=1,\dots,N} \abs{u_i - u(\bvec{x}_i)}}{\max_{i=1,\dots,N} u(\bvec{x}_i) - \min_{i=1,\dots,N} u(\bvec{x}_i)}.
\end{equation*}
All occurring linear systems are solved with the BiCGSTAB(2) method and an incomplete LU preconditioning.
We test the methods on Poisson's equation \eqref{eq:poisson_equation} in \cref{ssec:poisson_results} and the heat equation \eqref{eq:heat_equation_base} in \cref{ssec:heat_results}.

\subsection{Poisson's equation} \label{ssec:poisson_results}
First, we study the behavior of the new operator using Poisson's equation as in \cref{eq:poisson_equation}. We distinguish between problems with smooth diffusivity and discontinuous diffusivity.

Expanding the diffusion operator for a differentiable diffusivity $\kappa\in C^1(\Omega)$ and a sufficiently smooth analytic solution $u \in C^2(\Omega)$ as in \cref{eq:diffusion_expanded} allows us to define the source term $Q = -\nabla\cdot(\kappa\nabla u) = -\dotprod{\nabla\kappa}{\nabla u} - \kappa\Delta u$.
For the elliptic interface test case, we assume that $\kappa$ is piecewise constant and has a jump on a $d-1$ dimensional manifold $\Gamma\subset\Omega$. For a smooth function $f\in C^2(\Omega)$ with $f\vert_\Gamma = 0$, we define the analytic solution by $u = f/\kappa + c$ with a constant $c$. Because $f$ vanishes on the interface $\Gamma$, division by $\kappa$ introduces only a weak discontinuity. Hence, the resulting function is weakly differentiable, $u\in H^1(\Omega)$, with $\nabla u = \nabla f/ \kappa$ and thus $\kappa\nabla u = \nabla f$. As a consequence, the source term $Q = -\nabla\cdot(\kappa\nabla u) = -\Delta f$ is continuous.

\paragraph{Test case 1}
\begin{figure}
\begin{tikzpicture}
\begin{axis}[
    xlabel=$x_1$, xtick={0, 0.5, 1},
    ylabel=$x_2$, ytick={0, 0.5, 1},
    zlabel=$u$,
    width=\halfwidth,
    colormap name = viridis,
    z buffer=sort,
]
    \addplot3[only marks, mark size=1, scatter]
    table[x=x, y=y, z=poisson_smooth_solution, col sep=comma]{figures/data/testcases_05.csv};
\end{axis}
\end{tikzpicture}
\hfill
\begin{tikzpicture}
\begin{axis}[
    xlabel=$x_1$, xtick={0, 0.5, 1},
    ylabel=$x_2$, ytick={0, 0.5, 1},
    zlabel=$\kappa$,
    width=\halfwidth,
    colormap name = viridis,
    z buffer=sort,
]
    \addplot3[only marks, mark size=1, scatter]
    table[x=x, y=y, z=poisson_smooth_diffusivity, col sep=comma]{figures/data/testcases_05.csv};
\end{axis}
\end{tikzpicture}
\caption{Analytical solution (left) and diffusivity (right) for test case 1.}
\label{fig:setup_poisson_smooth}
\end{figure}

\begin{figure}
\centering
\begin{tikzpicture}
\begin{loglogaxis}[
    xlabel=$N$,
    ylabel=$\norm{u-u_h}_\infty$,
    cycle list name=wonglist,
    xmax=2e6,
    width=\halfwidth,
    legend style={fill=none, draw=none, font=\footnotesize},
    legend pos=south west,
    legend cell align=left,
]
    \addplot table[x=N, y=mls_2_2, col sep=comma]{figures/data/poisson_smooth.csv};
    \addlegendentry{MLS(2, 2)}
    \addplot table[x=N, y=mls_4_2, col sep=comma]{figures/data/poisson_smooth.csv};
    \addlegendentry{MLS(4, 2)}
    \addplot table[x=N, y=mls_4_4, col sep=comma]{figures/data/poisson_smooth.csv};
    \addlegendentry{MLS(4, 4)}

    \addplot[no marks, dotted, black, thick]
    table[x=N, y expr={1/x}, col sep=comma] {figures/data/poisson_smooth.csv}
    node[right] {$O(h^2)$};

    \addplot[no marks, dotted, black, thick]
    table[x=N, y expr={3/x^2}, col sep=comma] {figures/data/poisson_smooth.csv}
    node[right] {$O(h^4)$};
\end{loglogaxis}
\end{tikzpicture}
\caption{$L^\infty$ error for the solution of test case 1 with MLS based operators depending on the number of points $N$. The dotted lines depict the corresponding reference convergence rate.}
\label{fig:poisson_smooth_mls}
\end{figure}

\begin{figure}
\begin{subfigure}{\halfwidth}
\begin{tikzpicture}
\begin{loglogaxis}[
    xlabel=$N$,
    ylabel=$\norm{u-u_h}_\infty$,
    cycle list name=wonglist,
    xmax=2e6,
    ymin=1e-6,
    width=\textwidth,
    legend style={fill=none, draw=none, font=\footnotesize},
    legend pos=south west,
    legend cell align=left,
]

    \addplot table[x=N, y=ddo_2_am, col sep=comma]{figures/data/poisson_smooth.csv};
    \addplot table[x=N, y=ddo_2_hm, col sep=comma]{figures/data/poisson_smooth.csv};
    \addplot table[x=N, y=ddo_2_te, col sep=comma]{figures/data/poisson_smooth.csv};
    \addplot table[x=N, y=ddo_2_hi, col sep=comma]{figures/data/poisson_smooth.csv};

    \addlegendentry{DDO(2, AM)}
    \addlegendentry{DDO(2, HM)}
    \addlegendentry{DDO(2, TE)}
    \addlegendentry{DDO(2, HI)}

    \addplot[no marks, dotted, black, thick]
    table[x=N, y expr={3/x}, col sep=comma] {figures/data/poisson_smooth.csv}
    node[right] {$O(h^2)$};

\end{loglogaxis}
\end{tikzpicture}
\end{subfigure}
\hfil
\begin{subfigure}{\halfwidth}
\begin{tikzpicture}
\begin{loglogaxis}[
    xlabel=$N$,
    ylabel=$\norm{u-u_h}_\infty$,
    cycle list name=wonglist,
    xmax=2e6,
    ymin=1e-12,
    width=\textwidth,
    legend style={fill=none, draw=none, font=\footnotesize},
    legend pos=south west,
    legend cell align=left,
]

    \addplot table[x=N, y=ddo_4_am, col sep=comma]{figures/data/poisson_smooth.csv};
    \addplot table[x=N, y=ddo_4_hm, col sep=comma]{figures/data/poisson_smooth.csv};
    \addplot table[x=N, y=ddo_4_te, col sep=comma]{figures/data/poisson_smooth.csv};
    \addplot table[x=N, y=ddo_4_hi, col sep=comma]{figures/data/poisson_smooth.csv};

    \addlegendentry{DDO(4, AM)}
    \addlegendentry{DDO(4, HM)}
    \addlegendentry{DDO(4, TE)}
    \addlegendentry{DDO(4, HI)}

    \addplot[no marks, dotted, black, thick] table[x=N, y expr={0.2/x}, col sep=comma] {figures/data/poisson_smooth.csv} node[right] {$O(h^2)$};

    \addplot[no marks, dotted, black, thick] table[x=N, y expr={3/x^2}, col sep=comma] {figures/data/poisson_smooth.csv} node[right] {$O(h^4)$};
\end{loglogaxis}
\end{tikzpicture}
\end{subfigure}
\caption{$L^\infty$ error for the solution of test case 1 with derived diffusion operators depending on the number of points $N$. The dotted lines depict the corresponding reference convergence rate.}
\label{fig:poisson_smooth_ddo}
\end{figure}

Using the analytic solution $u(\bvec{x}) = \sin(\pi x_1) \sin(\pi x_2)$,
we choose $\kappa(\bvec{x}) = \exp(x_1 - x_2^2)$, displayed in \cref{fig:setup_poisson_smooth}.
The results for the different methods are summarized in \cref{fig:poisson_smooth_mls,fig:poisson_smooth_ddo}.
For the MLS ansatz in \cref{fig:poisson_smooth_mls}, we obtain second-order convergence for MLS(2, 2) and MLS(4, 2) while MLS(4, 4) yields fourth-order convergence.
This suggests that the MLS ansatz requires sufficiently accurate calculations of the diffusivity gradient.
It is however noteworthy that MLS(4, 2) provides reduced errors in comparison to MLS(2, 2).
On the other hand, the derived diffusion operator provides second-order convergence for all reconstructions if the underlying discrete Laplacian is second-order accurate.
We observe fourth-order convergence for the arithmetic averaging, harmonic averaging, and the Hermite interpolation with DDO(4, TE) being the only method limited to second-order convergence.
This is due to second-order gradients used in the reconstructions.
The overall results suggest that the dominant error of the derived diffusion operators in \cref{thm:diffusion_accuracy} stems from the discrete Laplace operator and not from the reconstructions.

\paragraph{Test case 2}
\begin{figure}
\centering
\begin{tikzpicture}
\begin{axis}[
    xlabel=$x_1$, xtick={0, 0.5, 1},
    ylabel=$x_2$, ytick={0, 0.5, 1},
    width=0.55\textwidth,
    colormap name = viridis,
    axis equal image,
    colorbar,
    point meta min=1,
    point meta max=3,
]
    \addplot[only marks, mark size=1, scatter]
    table[x=x, y=y, point meta={\thisrow{poisson_smooth_difficult_diffusivity}}, col sep=comma]{figures/data/testcases_05.csv};
\end{axis}
\end{tikzpicture}
\caption{Diffusivity $\kappa$ for test case 2.}
\label{fig:setup_poisson_smooth_difficult}
\end{figure}

\begin{figure}
\centering
\begin{tikzpicture}
\begin{loglogaxis}[
    xlabel=$N$,
    ylabel=$\norm{u-u_h}_\infty$,
    cycle list name=wonglist,
    xmax=2e6,
    width=\halfwidth,
    legend style={fill=none, draw=none, font=\footnotesize},
    legend pos=south west,
    legend cell align=left,
]
    \addplot table[x=N, y=mls_2_2, col sep=comma]{figures/data/poisson_smooth_difficult.csv};
    \addlegendentry{MLS(2, 2)}
    \addplot table[x=N, y=mls_4_2, col sep=comma]{figures/data/poisson_smooth_difficult.csv};
    \addlegendentry{MLS(4, 2)}
    \addplot table[x=N, y=mls_4_4, col sep=comma]{figures/data/poisson_smooth_difficult.csv};
    \addlegendentry{MLS(4, 4)}

    \addplot[no marks, dotted, black, thick] table[x=N, y expr={5/x}, col sep=comma] {figures/data/poisson_smooth_difficult.csv} node[right] {$O(h^2)$};

    \addplot[no marks, dotted, black, thick] table[x=N, y expr={300/x^2}, col sep=comma] {figures/data/poisson_smooth_difficult.csv} node[right] {$O(h^4)$};
\end{loglogaxis}
\end{tikzpicture}
\caption{$L^\infty$ error for the solution of test case 2 with MLS based operators depending on the number of points $N$. The dotted lines depict the corresponding reference convergence rate.}
\label{fig:poisson_smooth_difficult_mls}
\end{figure}

\begin{figure}
\begin{subfigure}{\halfwidth}
\begin{tikzpicture}
\begin{loglogaxis}[
    xlabel=$N$,
    ylabel=$\norm{u-u_h}_\infty$,
    cycle list name=wonglist,
    xmax=2e6,
    ymin=1e-6,
    width=\textwidth,
    legend style={fill=none, draw=none, font=\footnotesize},
    legend pos=south west,
    legend cell align=left,
]

    \addplot table[x=N, y=ddo_2_am, col sep=comma]{figures/data/poisson_smooth_difficult.csv};
    \addplot table[x=N, y=ddo_2_hm, col sep=comma]{figures/data/poisson_smooth_difficult.csv};
    \addplot table[x=N, y=ddo_2_te, col sep=comma]{figures/data/poisson_smooth_difficult.csv};
    \addplot table[x=N, y=ddo_2_hi, col sep=comma]{figures/data/poisson_smooth_difficult.csv};

    \addlegendentry{DDO(2, AM)}
    \addlegendentry{DDO(2, HM)}
    \addlegendentry{DDO(2, TE)}
    \addlegendentry{DDO(2, HI)}

    \addplot[no marks, dotted, black, thick]
    table[x=N, y expr={3/x}, col sep=comma] {figures/data/poisson_smooth_difficult.csv}
    node[right] {$O(h^2)$};

\end{loglogaxis}
\end{tikzpicture}
\end{subfigure}
\hfil
\begin{subfigure}{\halfwidth}
\begin{tikzpicture}
\begin{loglogaxis}[
    xlabel=$N$,
    ylabel=$\norm{u-u_h}_\infty$,
    cycle list name=wonglist,
    xmax=2e6,
    ymin=1e-10,
    width=\textwidth,
    legend style={fill=none, draw=none, font=\footnotesize},
    legend pos=south west,
    legend cell align=left,
]

    \addplot table[x=N, y=ddo_4_am, col sep=comma]{figures/data/poisson_smooth_difficult.csv};
    \addplot table[x=N, y=ddo_4_hm, col sep=comma]{figures/data/poisson_smooth_difficult.csv};
    \addplot table[x=N, y=ddo_4_te, col sep=comma]{figures/data/poisson_smooth_difficult.csv};
    \addplot table[x=N, y=ddo_4_hi, col sep=comma]{figures/data/poisson_smooth_difficult.csv};

    \addlegendentry{DDO(4, AM)}
    \addlegendentry{DDO(4, HM)}
    \addlegendentry{DDO(4, TE)}
    \addlegendentry{DDO(4, HI)}

    \addplot[no marks, dotted, black, thick] table[x=N, y expr={5/x}, col sep=comma] {figures/data/poisson_smooth_difficult.csv} node[right] {$O(h^2)$};

    \addplot[no marks, dotted, black, thick] table[x=N, y expr={300/x^2}, col sep=comma] {figures/data/poisson_smooth_difficult.csv} node[right] {$O(h^4)$};
\end{loglogaxis}
\end{tikzpicture}
\end{subfigure}
\caption{$L^\infty$ error for the solution of test case 2 with derived diffusion operators depending on the number of points $N$. The dotted lines depict the corresponding reference convergence rate.}
\label{fig:poisson_smooth_difficult_ddo}
\end{figure}

For this test case, we leave the analytic solution as in test case 1 and use a more challenging diffusivity $\kappa(\bvec{x}) = 2 + \sin(6\pi x_1) \sin(6\pi x_2)$, as shown in \cref{fig:setup_poisson_smooth_difficult}.
Because of the local extrema, the averaging reconstructions might fail to reconstruct the function values at the mid-points $\bvec{x}_{ij}$.
The results in \cref{fig:poisson_smooth_difficult_mls} let us draw similar conclusions as in test case 1 for the MLS method.
However, for the derived diffusion operators, we observe in \cref{fig:poisson_smooth_difficult_ddo} that the Hermite interpolation slightly outperforms the averaging reconstructions on the coarser point clouds.
However, the errors become comparable for the finest point cloud in the fourth-order case.
As before, the Taylor expansion reconstruction only shows second-order convergence.
In general, errors are higher for this test case in comparison to test case 1.

\paragraph{Test case 3}
\begin{figure}
\begin{tikzpicture}
\begin{axis}[
    xlabel=$x_1$, xtick={0, 0.5, 1},
    ylabel=$x_2$, ytick={0, 0.5, 1},
    zlabel=$u$,
    width=\halfwidth,
    colormap name = viridis,
    z buffer=sort,
]
    \addplot3[only marks, mark size=1, scatter]
    table[x=x, y=y, z=poisson_discontinuous_solution, col sep=comma]{figures/data/testcases_05.csv};
\end{axis}
\end{tikzpicture}
\hfill
\begin{tikzpicture}
\begin{axis}[
    xlabel=$x_1$, xtick={0, 0.5, 1},
    ylabel=$x_2$, ytick={0, 0.5, 1},
    zlabel=$\kappa$,
    width=\halfwidth,
    colormap name = viridis,
    z buffer=sort,
]
    \addplot3[only marks, mark size=1, scatter]
    table[x=x, y=y, z=poisson_discontinuous_diffusivity, col sep=comma]{figures/data/testcases_05.csv};
\end{axis}
\end{tikzpicture}
\caption{Analytical solution (left) and diffusivity (right) for test case 3.}
\label{fig:setup_poisson_discontinuous}
\end{figure}

\begin{figure}
\begin{subfigure}{\halfwidth}
\begin{tikzpicture}
\begin{loglogaxis}[
    xlabel=$N$,
    ylabel=$\norm{u-u_h}_\infty$,
    cycle list name=wonglist,
    xmax=2e6,
    width=\textwidth,
    legend style={fill=none, draw=none, font=\footnotesize},
    legend pos=south west,
    legend cell align=left,
]
    \addplot table[x=N, y=mls_2_2, col sep=comma]{figures/data/poisson_discontinuous.csv};
    \addlegendentry{MLS(2, 2)}

    \addplot table[x=N, y=ddo_2_te, col sep=comma]{figures/data/poisson_discontinuous.csv};
    \addlegendentry{DDO(2, TE)}

    \addplot table[x=N, y=ddo_2_hi, col sep=comma]{figures/data/poisson_discontinuous.csv};
    \addlegendentry{DDO(2, HI)}

    \addplot[no marks, dotted, black, thick]
    table[x=N, y expr={10/sqrt(x)}, col sep=comma] {figures/data/poisson_discontinuous.csv}
    node[right] {$O(h)$};

\end{loglogaxis}
\end{tikzpicture}
\end{subfigure}
\hfil
\begin{subfigure}{\halfwidth}
\begin{tikzpicture}
\begin{loglogaxis}[
    xlabel=$N$,
    ylabel=$\norm{u-u_h}_\infty$,
    cycle list name=wonglist,
    xmax=2e6,
    width=\textwidth,
    legend style={fill=none, draw=none, font=\footnotesize},
    legend pos=south west,
    legend cell align=left,
]

    \addplot table[x=N, y=ddo_2_am, col sep=comma]{figures/data/poisson_discontinuous.csv};
    \addplot table[x=N, y=ddo_2_hm, col sep=comma]{figures/data/poisson_discontinuous.csv};

    \addlegendentry{DDO(2, AM)}
    \addlegendentry{DDO(2, HM)}

    \addplot[no marks, dotted, black, thick]
    table[x=N, y expr={1/sqrt(x)}, col sep=comma] {figures/data/poisson_discontinuous.csv}
    node[right] {$O(h)$};

\end{loglogaxis}
\end{tikzpicture}
\end{subfigure}
\caption{$L^\infty$ error for the solution of test case 3 depending on the number of points $N$. The dotted lines depict the corresponding reference convergence rate.}
\label{fig:poisson_discontinuous}
\end{figure}

Now we consider a test case with a discontinuous diffusivity $\kappa$. With $c = \frac{3}{4}$, we choose $f(\bvec{x}) = \sin(\pi x_1) \sin(\pi x_2) - c$ and a piecewise constant diffusivity with a jump of eight orders of magnitude
\[
    \kappa(\bvec{x}) = \begin{cases}
        10^8, & f(\bvec{x}) \ge 0, \\
        1,    & f(\bvec{x}) < 0.
    \end{cases}
\]
The resulting interface $\Gamma = \set{\bvec{x} \mid f(\bvec{x}) = 0}$ can be seen in \cref{fig:setup_poisson_discontinuous} along which $u$ has a weak discontinuity and $\kappa$ has a jump.
For this test case, we restrict ourselves to consistency conditions for monomials of up to second degree because we we will not exceed first-order convergence.

The results on the left-hand side in \cref{fig:poisson_discontinuous} show methods that do not converge for test case 3, namely MLS(2, 2), and the derived diffusion operator with gradient-based reconstructions DDO(2, TE) and DDO(2, HI).
While DDO(2, HI) provides good results for some point clouds, it did not perform reliably for the entire suite of point clouds.
The main reason for these reconstructions failing are overshoots and undershoots resulting in negative reconstructions $\kappa_{ij} < 0$ for some points.
Negative reconstructions directly violate the diagonal dominance of the derived diffusion operator leading to a lack of stability conditions of the numerical scheme.
Negative reconstructions can be circumvented by other formulations of discrete gradients that are less susceptible to jumps, for example WENO techniques \cite{Friedrich_1998}.
This leads to a weak notion of the gradient and, in this example, to vanishing discrete gradients $\nabla_i\kappa\approx\bvec{0}$.
Thus, the Hermite interpolation would reduce to the arithmetic averaging.

The arithmetic averaging DDO(2, AM) and the harmonic averaging DDO(2, HM) provide first-order convergence due to the bound $\min(\kappa_i, \kappa_j) \le \kappa_{ij} \le \max(\kappa_i, \kappa_j)$ as can be seen on the right-hand side of \cref{fig:poisson_discontinuous}.
The harmonic averaging clearly outperforms the arithmetic averaging in this test case.
This has also been observed for other interface problems that are not presented in this work.
Higher-order convergence can only be achieved by explicitly resolving the interface by adding points and performing a domain decomposition as has been previously presented by Davydov and Safarpoor \cite{Davydov_Safarpoor_2021}.

\subsection{Heat equation} \label{ssec:heat_results}
To solve the heat equation \eqref{eq:heat_equation_base}, we multiply the analytic solutions from \cref{ssec:poisson_results}, which we call $\bar{u}$, with a time-dependent function $a$ such that $u(\bvec{x}, t) = a(t) \bar{u}(\bvec{x})$ is an analytic solution to the heat equation with the source term $Q(\bvec{x}, t) = a'(t) \bar{u}(\bvec{x}) - a(t) \nabla\cdot(\kappa\nabla) \bar{u}(\bvec{x}).$

Suchde \cite{Suchde_2018} showed for a similar problem with a constant diffusivity $\kappa = 1$ that the absence of diagonally dominant operators leads to severe instabilities in the numerical solution after a few time steps.
Similarly to the test case used therein, we define $a(t) = \exp(-4t)$ for all test cases in this section and simulate until $t=1$.
Doing so, we measure the maximum relative $L^\infty$ error in the iteration process.

The time step size $\dt$ used in the implicit trapezoidal rule fulfills the CFL condition
\begin{equation} \label{eq:CFL_condition}
    \dt \le 0.7 \Delta x^2,
\end{equation}
where $\Delta x$ is the smallest distance between two points in the point cloud.
Since the implicit trapezoidal rule is A-stable, the CFL condition is not necessary for ensuring stability of the numerical scheme.
But since it is not positivity preserving, imposing a CFL condition can help to ensure positivity of the numerical solutions.

\paragraph{Test case 4}
\begin{figure}
\begin{subfigure}{\halfwidth}
\begin{tikzpicture}
\begin{loglogaxis}[
    xlabel=$N$,
    ylabel=$\norm{u-u_h}_\infty$,
    cycle list name=wonglist,
    xmax=3e4,
    ymin=5e-9,
    width=\textwidth,
    legend style={fill=none, draw=none, font=\footnotesize},
    legend pos=south west,
    legend cell align=left,
]
    \addplot table[x=N, y=mls_2_2, col sep=comma]{figures/data/heat_smooth.csv};
    \addlegendentry{MLS(2, 2)}

    \addplot table[x=N, y=mls_4_2, col sep=comma]{figures/data/heat_smooth.csv};
    \addlegendentry{MLS(4, 2)}

    \addplot table[x=N, y=mls_4_4, col sep=comma]{figures/data/heat_smooth.csv};
    \addlegendentry{MLS(4, 4)}

    \addplot[no marks, dotted, black, thick]
    table[x=N, y expr={1/x}, col sep=comma] {figures/data/heat_smooth.csv}
    node[right] {$O(h^2)$};

    \addplot[no marks, dotted, black, thick]
    table[x=N, y expr={3/x^2}, col sep=comma] {figures/data/heat_smooth.csv}
    node[right] {$O(h^4)$};

\end{loglogaxis}
\end{tikzpicture}
\end{subfigure}
\hfil
\begin{subfigure}{\halfwidth}
\begin{tikzpicture}
\begin{loglogaxis}[
    xlabel=$N$,
    ylabel=$\norm{u-u_h}_\infty$,
    cycle list name=wonglist,
    xmax=3e4,
    ymin=5e-11,
    width=\textwidth,
    legend style={fill=none, draw=none, font=\footnotesize},
    legend pos=south west,
    legend cell align=left,
]

    \addplot table[x=N, y=ddo_2_am, col sep=comma]{figures/data/heat_smooth.csv};
    \addplot table[x=N, y=ddo_2_hm, col sep=comma]{figures/data/heat_smooth.csv};
    \addplot table[x=N, y=ddo_4_am, col sep=comma]{figures/data/heat_smooth.csv};
    \addplot table[x=N, y=ddo_4_hm, col sep=comma]{figures/data/heat_smooth.csv};

    \addlegendentry{DDO(2, AM)}
    \addlegendentry{DDO(2, HM)}
    \addlegendentry{DDO(4, AM)}
    \addlegendentry{DDO(4, HM)}

    \addplot[no marks, dotted, black, thick]
    table[x=N, y expr={3/x}, col sep=comma] {figures/data/heat_smooth.csv}
    node[right] {$O(h^2)$};

    \addplot[no marks, dotted, black, thick]
    table[x=N, y expr={3/x^2}, col sep=comma] {figures/data/heat_smooth.csv}
    node[right] {$O(h^4)$};

\end{loglogaxis}
\end{tikzpicture}
\end{subfigure}
\caption{$L^\infty$ error for the solution of test case 4 depending on the number of points $N$. The dotted lines depict the corresponding reference convergence rate.}
\label{fig:heat_smooth}
\end{figure}

In this example, we use the analytic solution from test case 1 as $\bar{u}$.
In the results in \cref{fig:heat_smooth} we have omitted the gradient-based reconstructions since the results are generally very similar to the results of test case 2.
We observe that MLS(2, 2) and MLS(2, 4) are second-order accurate while MLS(4, 4) is fourth-order accurate.
Both reconstructions of the derived diffusion operators conserve the order of the underlying discrete Laplace operator.

\paragraph{Test case 5}
\begin{figure}
\centering
\begin{tikzpicture}
\begin{loglogaxis}[
    xlabel=$N$,
    ylabel=$\norm{u-u_h}_\infty$,
    cycle list name=wonglist,
    xmax=5e4,
    ymin=1e-5,
    width=\halfwidth,
    legend style={fill=none, draw=none, font=\footnotesize},
    legend pos=south west,
    legend cell align=left,
]

    \addplot table[x=N, y=mls_2_2, col sep=comma]{figures/data/heat_discontinuous.csv};
    \addplot table[x=N, y=ddo_2_hm, col sep=comma]{figures/data/heat_discontinuous.csv};

    \addlegendentry{MLS(2, 2)}
    \addlegendentry{DDO(2, HM)}

    \addplot[no marks, dotted, black, thick]
    table[x=N, y expr={0.1/sqrt(sqrt(x))}, col sep=comma] {figures/data/heat_discontinuous.csv}
    node[right] {$O(h^{1/2})$};

\end{loglogaxis}
\end{tikzpicture}
\caption{$L^\infty$ error for the solution of test case 5 depending on the number of points $N$. The dotted lines depict the corresponding reference convergence rate.}
\label{fig:heat_discontinuous}
\end{figure}

Similarly to test case 4, we use the analytic solution for the elliptic interface problem from test case 3 as $\bar{u}$.
For this final test case, we have selected harmonic averaging for the derived diffusion operator since it performed the best in test case 3.
Additionally, we added MLS(2, 2) to test if the MLS ansatz is a viable option for time-dependent problems.

The results in \cref{fig:heat_discontinuous} show that for this test case, we cannot reproduce the behavior as seen for test case 3.
While MLS(2, 2) seems to converge, it produces very high errors compared to DDO(2, HM) which only shows $1/2$-order convergence.
It is possible that the first-order convergence would only be visible for coarser point clouds and that for the considered point clouds, we have converged to a final error.
Nevertheless, the derived diffusion operator is superior to the MLS based calculation of the discrete diffusion operator for elliptic and parabolic interface problems.

\section{Final remarks} \label{sec:conclusion}
In this paper, we derived a new discretization method of the diffusion operator based on weighting the discrete Laplace operator by using reconstruction functions.
We proved that the discrete diffusion operator preserves the consistency order from the discrete Laplace operator with additional error terms that arise from the reconstruction.
Furthermore, we presented a possibility of extending the derived operator idea to anisotropic diffusion operators which might be featured in future research.
Also, the study and comparison of different linear solvers and preconditioners in combination with different reconstructions is necessary to be able to solve the occurring linear systems more efficiently.

We tested the derived diffusion operator and showed that even second-order reconstructions preserve the fourth-order accuracy of the discrete Laplace operator for problems with a smooth diffusivity.
For interface problems, we demonstrated the applicability of the new discrete diffusion operator, and numerically showed first-order accuracy for elliptic interface problems.

\section*{Acknowledgments}
Pratik Suchde would like to acknowledge support from the European Union's Horizon 2020
research and innovation program under the Marie Skłodowska-Curie Actions grant agreement No. 892761. Pratik Suchde would like to acknowledge funding from the Institute of Advanced Studies, University of Luxembourg, under the AUDACITY program.

\bibliographystyle{ieeetr}

\end{document}